\documentclass[12pt,reqno]{article}

\usepackage[usenames]{color}
\usepackage{amssymb}
\usepackage{amsmath}
\usepackage{amsthm}
\usepackage{amsfonts}
\usepackage{amscd}
\usepackage{graphicx}

\usepackage[colorlinks=true,
linkcolor=webgreen,
filecolor=webbrown,
citecolor=webgreen]{hyperref}

\definecolor{webgreen}{rgb}{0,.5,0}
\definecolor{webbrown}{rgb}{.6,0,0}

\usepackage{color}
\usepackage{fullpage}
\usepackage{float}

\usepackage{graphics}
\usepackage{latexsym}
\usepackage{epsf}
\usepackage{breakurl}

\def\modd#1 #2{#1\ \mbox{\rm (mod}\ #2\mbox{\rm )}}

\def\Enn{\mathbb{N}}

\begin{document}

\theoremstyle{plain}
\newtheorem{theorem}{Theorem}
\newtheorem{corollary}[theorem]{Corollary}
\newtheorem{lemma}[theorem]{Lemma}
\newtheorem{proposition}[theorem]{Proposition}

\theoremstyle{definition}
\newtheorem{definition}[theorem]{Definition}
\newtheorem{example}[theorem]{Example}
\newtheorem{conjecture}[theorem]{Conjecture}

\theoremstyle{remark}
\newtheorem{remark}[theorem]{Remark}

\title{Robbins and Ardila meet Berstel}

\author{Jeffrey Shallit\\
School of Computer Science\\
University of Waterloo\\
Waterloo, ON  N2L 3G1 \\
Canada\\
\href{mailto:shallit@uwaterloo.ca}{\tt shallit@uwaterloo.ca}}

\maketitle

\newenvironment{smallarray}[1]
{\null\,\vcenter\bgroup\scriptsize
\renewcommand{\arraystretch}{0.7}%
\arraycolsep=.13885em
\hbox\bgroup$\array{@{}#1@{}}}
{\endarray$\egroup\egroup\,\null}

\begin{abstract}
In 1996, Neville Robbins proved the amazing fact
that the coefficient of
$X^n$ in the Fibonacci infinite product
$$ \prod_{n \geq 2} (1-X^{F_n}) = (1-X)(1-X^2)(1-X^3)(1-X^5)(1-X^8) \cdots
= 1-X-X^2+X^4 + \cdots$$
is always either $-1$, $0$, or $1$.
The same result was proved later by
Federico Ardila using a different method.

Meanwhile, in 2001, Jean Berstel gave a simple 4-state transducer that
converts an ``illegal'' Fibonacci representation into a ``legal'' one.
We show how to obtain the Robbins-Ardila result from Berstel's with almost
no work at all, using purely computational techniques that can
be performed by existing software.
\end{abstract}

\section{Introduction}

The goal of this paper is to show how to prove an amazing 1996 result
of Robbins \cite{Robbins:1996} on the coefficients of a Fibonacci
infinite product, namely that the coefficient of $X^n$ in
$$ \prod_{n \geq 2} (1-X^{F_n}) = (1-X)(1-X^2)(1-X^3)(1-X^5)(1-X^8) \cdots
= 1-X-X^2+X^4 + \cdots$$
is always either $-1, 0, $ or $1$ \cite{Robbins:1996}.  A different
proof was given later by Ardila \cite{Ardila:2004}.  The novelty of
our approach is that it is purely ``computational'',
using algebraic techniques on automata
that can be carried out by existing software, starting
from a construction of Jean Berstel.  With this approach one can also
prove new results (see Section~\ref{new}).

\section{Fibonacci representation}

Let us start with the basics of Fibonacci representation
(also known as Zeckendorf representation) \cite{Lekkerkerker:1952,Zeckendorf:1972}.
Every natural number has an essentially unique representation
as a sum of 
Fibonacci numbers $n = \sum_{0 \leq i < t} e_i F_{i+2}$,
provided that no two consecutive Fibonacci numbers are used.  (Here, as usual,
we write $F_0 = 0$, $F_1 = 1$, and $F_n = F_{n-1} + F_{n-2}$ for
$n \geq 2$.)
If $n$ is written this way,
we define $(n)_F$ to be $e_{t-1} e_{t-2} \cdots e_0$, a binary
string called the {\it canonical
Fibonacci representation} of $n$.   The map $n \rightarrow (n)_F$ gives
a bijection between $\Enn$ and the strings specified by the regular
expression $C_F := \epsilon + 1(0+01)^*$---that is,
the set of all binary
having no two consecutive $1$'s that do not start with $0$.
We also define, for a binary string $x = b_1 \cdots b_t$,
the map $[n]_F := \sum_{1 \leq i \leq t} b_i F_{t-i+2}$.
Thus, for example, $(43)_F = 10010001$ and $[0010001101]_F = 43$.  

\section{From Berstel's transducer to a linear representation}

We start with Berstel's transducer \cite{Berstel:2001}.
When rewritten as a DFA $M$,
it becomes the following:
\begin{figure}[H]
\begin{center}
\includegraphics[width=4in]{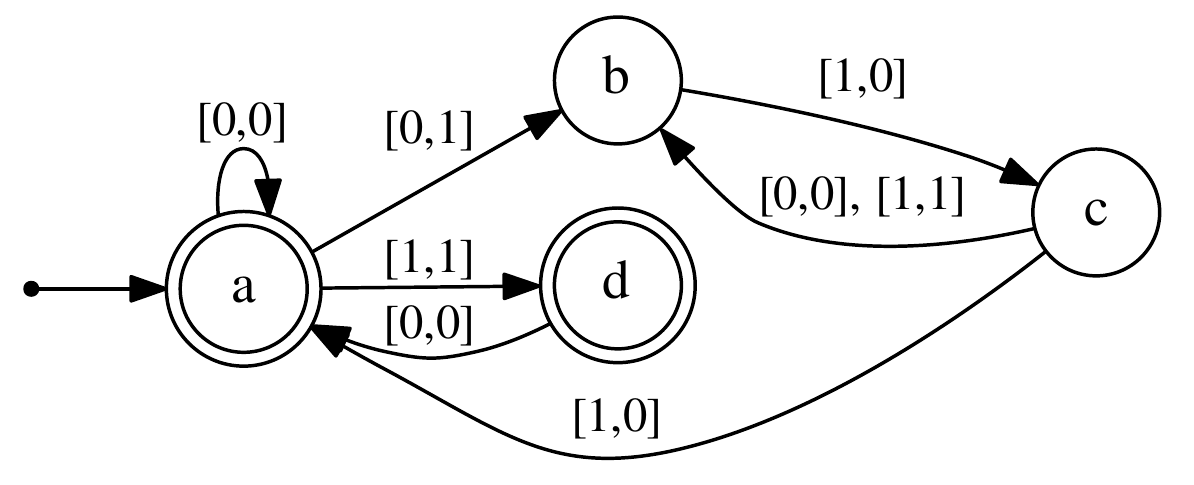}
\end{center}
\caption{Berstel's DFA $M$ for Fibonacci normalization.}
\label{berstel-aut}
\end{figure}
Inputs to $M$ consist of strings of pairs of letters.
The automaton $M$ accepts if the string spelled out by the first
components of the input---an arbitrary string of $0$'s and $1$'s---evaluates
to the same number as the the canonical Fibonacci representation spelled
out by the second components.  (Here, as usual, accepting states
are denoted by double circles.)
More formally, if $x = x_1 \cdots x_i$ and
$y = y_1 \cdots y_i$, define $x \times y$ to
be the string of pairs $[x_1, y_1] \cdots [x_i, y_i]$.
On input $x \times y$,
the automaton accepts iff
\begin{equation} 
x\in \{ 0, 1\}^* \text{ and } y \in 0^*C_F \text{ and }
[x]_F = [y]_F. \label{cond}
\end{equation}

Now suppose $y$ is a canonical Fibonacci representation for
$n$.   Let us count the number $r(n)$ of strings $x$ such that
$M$ accepts $x \times y$.   As Berstel observed, this is 
the number of binary strings $x$ such that $|x|=|y|$ and $[x]_F = [y]_F$.
In other words, this is the number of {\it Fibonacci partitions}
of $n$:  the number of ways to write $n$ as a sum of Fibonacci
numbers, where order does not matter.
For example, $r(8) = 3$, corresponding to
the three accepted strings 
$$ [1,1][0,0][0,0][0,0][0,0], \ 
[0,1][1,0][1,0][0,0][0,0],\ 
[0,1][1,0][0,0][1,0][1,0] $$
and the three Fibonacci partitions $8 = 5+3 = 5+2+1$.

If we now define $(\mu(a))_{i,j}$ as the number of paths
labeled $[*,a]$ (where the star $*$ means any symbol)
from state $i$ to state $j$ of $M$, we get
a so-called {\it linear representation\/} $(v, \mu, w)$ for $r(n)$:
\begin{align*}
v &= \left[ \begin{array}{cccc} 
1 & 0 & 0 & 0
\end{array} \right]; \quad &
\mu(0) &= \left[ \begin{array}{cccc}
1 & 0 & 0 & 0 \\
0 & 0 & 1 & 0 \\
1 & 1 & 0 & 0 \\
1 & 0 & 0 & 0
\end{array}\right]; \quad &
\mu(1) &= \left[ \begin{array}{cccc}
0 & 1 & 0 & 1 \\
0 & 0 & 0 & 0 \\
0 & 1 & 0 & 0 \\
0 & 0 & 0 & 0
\end{array}\right]; \quad &
w &= \left[ \begin{array}{c}
1 \\
0 \\
0 \\
1
\end{array}\right].
\end{align*}
Here $\mu$ is a morphism, that is, a map satisfying
$\mu(x)\mu(y) = \mu(xy)$ for all strings $x$ and $y$.
If $(n)_F = x$, then $r(n) = v \mu(x) w$.  This gives a
very efficient way to compute $r(n)$:  write $n$ as its canonical Fibonacci
representation $x$, multiply the matrices $\mu(0)$ and $\mu(1)$ according to the
bits of $x$, and then pre- and post-multiply by the
vectors $v$ and $w$.  The {\it rank\/} of a linear representation $(v, \mu, w)$
is the dimension of the vector $v$; in this case it is $4$.

Notice that $r(n)$ is just the coefficient of $X^n$ in the
following Fibonacci power series:
$$ \prod_{i \geq 2} (1+X^{F_n}) = (1+X)(1+X^2)(1+X^3)(1+X^5)(1+X^8) \cdots .$$
Of course, $r(n)$ is unbounded.

Robbins took this power series and modified it to
\begin{align}
\prod_{i \geq 2} (1-X^{F_n}) &= (1-X)(1-X^2)(1-X^3)(1-X^5)(1-X^8) \cdots 
	\nonumber \\
&= \sum_{n \geq 0} a(n) X^n  ,  \label{robb}
\end{align} so that
$a(n)$ is the coefficient of $X^n$ in this new series.
He observed that if $r_e(n)$ is the number
of Fibonacci partitions using an even number of terms, and
$r_o (n)$ is the number of Fibonacci partitions using an odd
number of terms, then clearly $r(n) = r_e(n) + r_o(n)$.  Furthermore,
Robbins noted that Eq.~\eqref{robb} gives $a(n) = r_e (n) - r_o (n)$.
By adding these two equations we get
$a(n) = 2r_e(n) - r(n)$.
Since we already know how to compute $r(n)$, to compute $a(n)$
we only need to know $r_e (n)$.

We can find a linear representation for $r_e(n)$ using a trivial
modification of Berstel's automaton.   It suffices to create a new
automaton $M'$ accepting those pairs $x \times y$ exactly as
before, but constrained by the number of $1$'s in $x$ being even.
This amounts to performing a cross product construction of $M$ with the
following simple automaton, where again $*$ matches any symbol:
\begin{figure}[H]
\begin{center}
\includegraphics[width=2.5in]{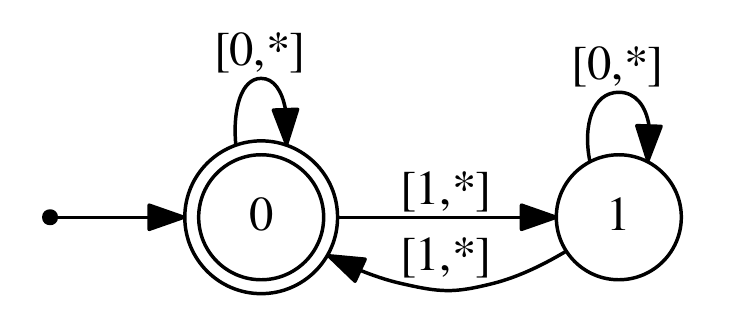}
\end{center}
\caption{Automaton for even number of $1$'s in the first component.}
\end{figure}
\noindent This cross product can be
computed automatically with software that manipulates
automata, such as {\tt Grail} \cite{Raymond&Wood:1994}.

This gives the
new automaton $M'$ below, which accepts those inputs over the alphabet
$\Sigma_2^*$ that satisfy
the condition \eqref{cond} and also have an even number of $1$'s
in the first components.  The names of the states match those
in Fig.~\ref{berstel-aut}, together with the parity ($0$ or $1$) of
the number of $1$'s in the first coordinate.
\begin{figure}[H]
\begin{center}
\includegraphics[width=6.5in]{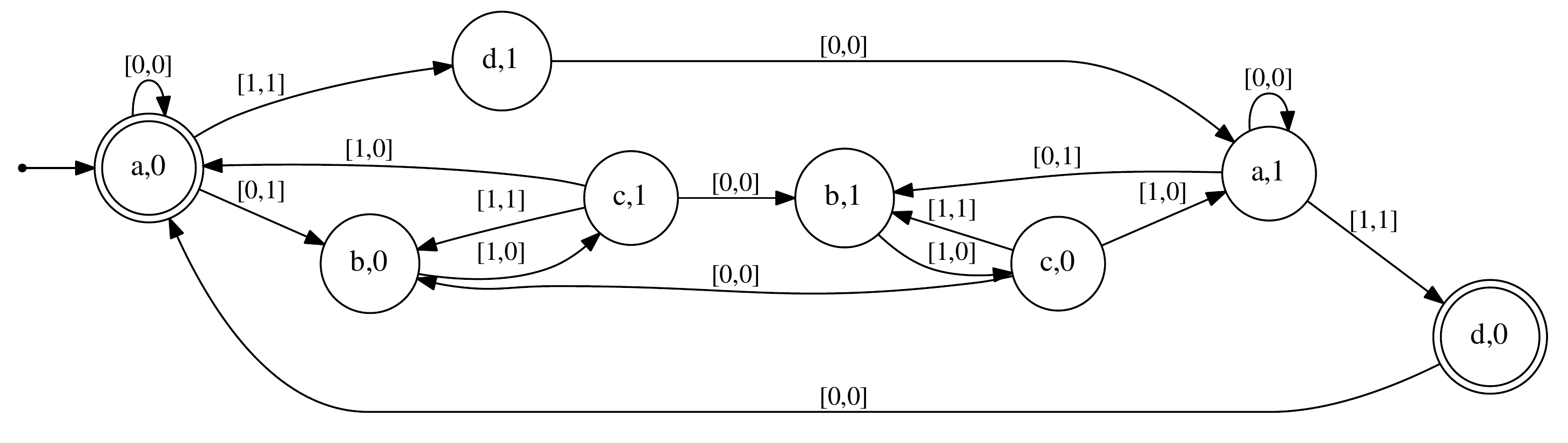}
\end{center}
\caption{$M'$:  Berstel's automaton modified.}
\label{berstel-aut2}
\end{figure}

The next step is to find the linear representation 
corresponding to the automaton $M'$.
It is $(v', \mu', w')$, as given below.
Again, this can be computed ``automatically'' just by counting paths
in the transition diagram of $M'$.
\begin{align*}
v &= \left[ \begin{smallarray}{cccccccc} 
1&0&0&0&0&0&0&0
\end{smallarray} \right]; \quad &
\mu'(0) &= \left[ \begin{smallarray}{cccccccc}
1&0&0&0&0&0&0&0\\
0&1&0&0&0&0&0&0\\
0&0&0&0&0&1&0&0\\
0&0&0&0&1&0&0&0\\
0&1&1&0&0&0&0&0\\
1&0&0&1&0&0&0&0\\
1&0&0&0&0&0&0&0\\
0&1&0&0&0&0&0&0
\end{smallarray}\right]; \quad &
\mu'(1) &= \left[ \begin{smallarray}{cccccccc}
0&0&1&0&0&0&0&1\\
0&0&0&1&0&0&1&0\\
0&0&0&0&0&0&0&0\\
0&0&0&0&0&0&0&0\\
0&0&0&1&0&0&0&0\\
0&0&1&0&0&0&0&0\\
0&0&0&0&0&0&0&0\\
0&0&0&0&0&0&0&0
\end{smallarray}\right]; \quad &
w' &= \left[ \begin{smallarray}{c}
1\\
0\\
0\\
0\\
0\\
0\\
1\\
0
\end{smallarray}\right].
\end{align*}
This representation has rank $8$.

One can also use the program {\tt Walnut}, written by
Hamoon Mousavi \cite{Mousavi:2016}, to produce
the linear representation directly.  We give the details now.
First, we write a {\tt Walnut} regular expression specifying that $x$ has an
even number of $1$'s:
\begin{verbatim}
reg even1 "0*(10*10*)":
\end{verbatim}

Next, after having stored the automaton in Figure~\ref{berstel-aut} in
{\tt Walnut} format in the directory {\tt Automata Library}, under
the name {\tt \$berst}, we issue the following command
\begin{verbatim}
eval fibeven y "?msd_fib $even1(x) & $berst(x,y)":
\end{verbatim}
The linear representation can then be found in the file
{\tt fibeven.mpl} in the {\tt Result} directory.

Next, we can construct a linear representation for the function
$a(n) = 2 r_e (n) - r(n)$, by just combining the ones for 
$r_e$ and $r$:
\begin{align*}
v'' &= \left[ \begin{array}{cc} 
2v & -v'
\end{array} \right]; &
\mu''(0) &= \left[ \begin{array}{cc}
\mu(0) & 0 \\
0 & \mu'(0)
\end{array}\right];  &
\mu''(1) &= \left[ \begin{array}{cc}
\mu(1) & 0 \\
1 & \mu'(1)
\end{array}\right];  &
w'' &= \left[ \begin{array}{c}
w \\
w'
\end{array}\right].
\end{align*}
This gives us a linear representation for $a(n)$ of rank $12$.

Next, we minimize this linear representation, using the algorithm
given by Berstel and Reutenauer \cite{Berstel&Reutenauer:2010}.
We get the equivalent rank-$4$ linear representation $(y,\gamma, z)$
for $a(n)$ below:
\begin{align*}
y &= \left[ \begin{array}{cccc} 
1&0&0&0
\end{array} \right]; &
\gamma(0) &= \left[ \begin{array}{cccc}
1& 0& 0& 0 \\
0& 0& 1& 0 \\
0& 0& 0& 1 \\
-1& 0&-1& 0
\end{array}\right];  &
\gamma(1) &= \left[ \begin{array}{cccc}
0& 1& 0& 0 \\
0& 0& 0& 0 \\
0&-1& 0&-1\\
0& 0& 0& 0
\end{array}\right];  &
z &= \left[ \begin{array}{c}
1 \\
-1\\
-1\\
0
\end{array}\right].
\end{align*}

\section{Finishing up the proof}

Finally, we can use breadth-first search
(aka the ``semigroup trick'' of
\cite{Du&Mousavi&Schaeffer&Shallit:2016}) to verify that the set of all
products of the form $y \, \gamma(x)$, $x \in \{ 0, 1 \}^*$ is finite.
We find that the resulting semigroup $S$ is of cardinality $15$.
(We remark that the semigroup generated by the
two matrices $\gamma(0)$ and $\gamma(1)$ has cardinality $207$.)
Each member of $S$ is a vector, and we can easily check that
the dot product of each vector with $z$ gives only $0, 1, -1$.
The result of Robbins is now
proved.   Furthermore,
the linear representation $(y, \gamma, z)$
provides a simple algorithm to compute $a(n)$.

\section{Going further} 

We can prove even more.   The semigroup $S = \{ y \, \gamma(x) \, : \, 
x\in \{0, 1 \}^*\}$ allows us to construct a finite automaton $A$ that
computes $a(n)$ in the following
way:  on input the canonical Fibonacci representation
$(n)_F$, the automaton arrives at a state with output $a(n)$.
Here the 
states are named $y \, \gamma(x)$ for some $x$,
the initial state is $y$, transitions are given
by $\delta(u, a) = u \cdot \gamma(a)$,
and the output of the state named
$y \, \gamma(x)$ is $y\, \gamma(x) z$.
The automaton $A$ is hence algorithmically constructible,
and the result is displayed below:
\begin{figure}[H]
\begin{center}
\includegraphics[width=5.4in]{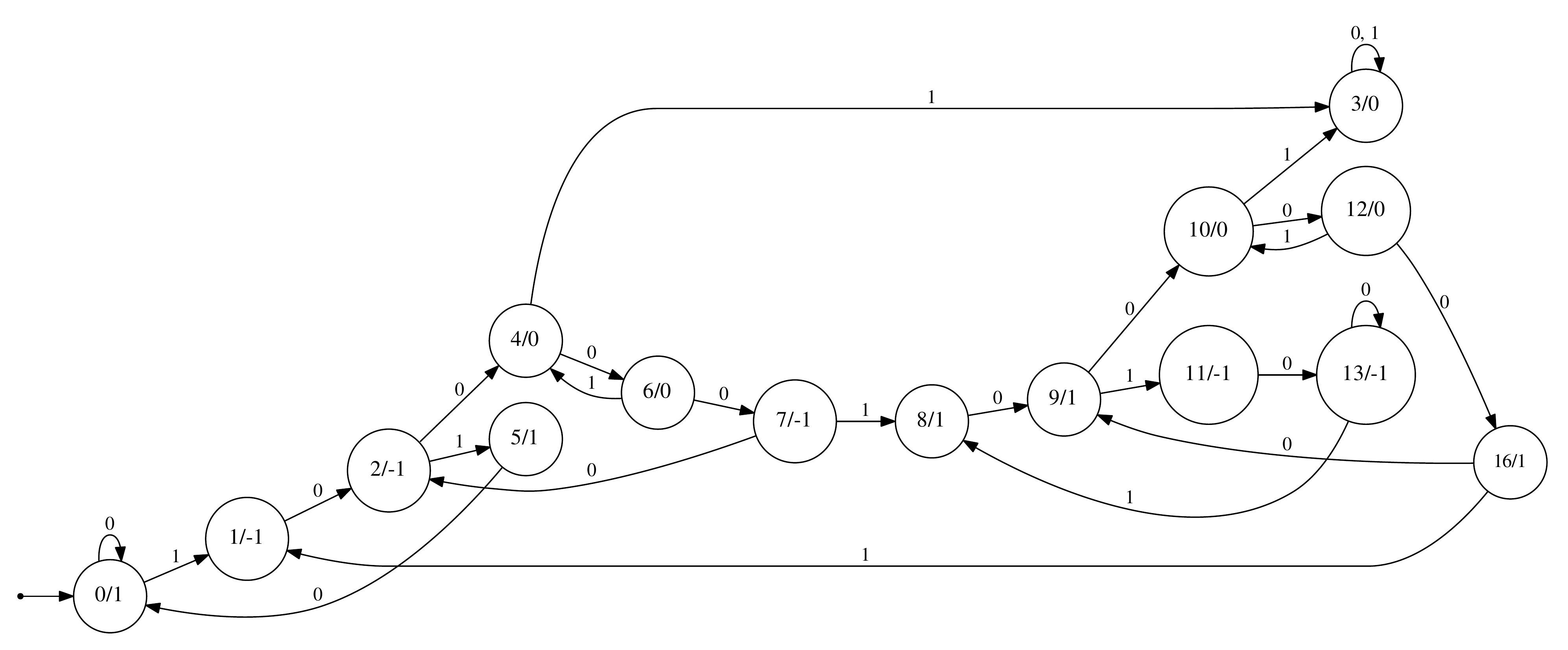}
\end{center}
\caption{The automaton $A$ computing $a(n)$}
\end{figure}
The automaton $A$ gives us a lot of information about how
$a(n)$ behaves.   For example, Ardila proved that ``almost all''
$n$ have $a(n) = 0$.  We can easily prove this using the DFA $A$
as follows:   clearly, almost all natural numbers $n$ have a Fibonacci
representation containing the block $t = 01001001$.  Now all
we have to check is that $t$ is a {\it synchronizing word} (see \cite{Volkov:2008})
for $A$:  the action of $t$ on every state $q$ maps $q$
to state $3$, which has output $0$.

Furthermore, the representations we have obtained for
$r(n)$, $r_e(n)$, and $a(n)$ give us most of the other
results of Robbins, without the need for inductions or case analysis.
\begin{theorem}
\leavevmode
\begin{itemize}
\item[(a)]  $r(F_n) = \lfloor n/2 \rfloor$ for $n \geq 2$.
\item[(b)] $r_e(F_n) = \lfloor n/4 \rfloor$ for $n \geq 1$.
\item[(c)] $a(F_n - 1) = \begin{cases}
	1, & \text{if $n \equiv \modd{1,2} {4}$;}\\
	-1, & \text{if $n \equiv \modd{0,3} {4}$;}
	\end{cases}$ and $n \geq 1$.
\end{itemize}
\end{theorem}

\begin{proof}
\leavevmode
\begin{itemize}
\item[(a)]
Since $(F_n)_F = 1 \, 0^{n-2}$, it follows that
$r(F_n) = v \mu(1) \mu(0)^{n-2} w$.  Hence $r(F_n)$
is a linear combination of the entries of $\mu(0)^{n-2}$.
But each entry of the matrices $\mu(0)^t$, considered as a sequence
indexed by $t$, satisfies a linear
recurrence whose annihilating polynomial is the minimal
polynomial of $\mu(0)$, and hence so do the values $r(F_n)$.
We can use computer algebra
software, such as {\tt Maple}, to compute this minimal
polynomial; it is $X(X+1)(X-1)^2$.  By
the fundamental theorem of linear recurrences we know
that $r(F_n) = c_1 n + c_2 + c_3 (-1)^n$ for $n \geq 2$.  Solving
for the constants gives us $c_1 = 1/2$, $c_2 = -1/4$, and
$c_3 = 1/4$.   Hence $r(F_n) = \lfloor n/2 \rfloor$ for 
$n \geq 2$, as desired.

\item[(b)] Here we play the same game, but for the linear
representation $(v', \mu', w')$.   We get a minimal
polynomial of $X(X+1)(X^2+1)(X-1)^2$ for $\mu'(0)$.   Hence
$r_e(F_n) = c_1 n + c_2 + c_3 (-1)^n + c_4 i^n + c_5 (-i)^n$,
where $i = \sqrt{-1}$.   Solving for the constants
gives $c_1 = 1/4$, $c_2 = -3/8$, $c_3 = 1/8$, $c_4 = 1/8-i/8$,
$c_5 = 1/8 + i/8$, which gives the desired result.

\item[(c)]  The Fibonacci representation for $F_n - 1$ is given
by $(10)^{n/2 - 1}$  if $n \geq 2$ is even, and $(10)^{(n-3)/2} 1$ if
$n \geq 3$ is odd.   Examining the path in the automaton $A$ labeled
by $101010\cdots$, the result follows immediately.
\end{itemize}
\end{proof}

The same ideas can be used to easily prove the 
equality $r(F_n^2-1) = F_n$ for $n \geq 2$ from \cite{Stockmeyer:2008} and
to prove the theorems in \cite{Bicknell-Johnson&Fielder:1999}.

\section{Some new results}
\label{new}

Another advantage to this method is that with exactly the same
techniques, we can go on to study
additional variations on the Fibonacci infinite product.  For example, 
we could study Fibonacci partitions with the number of parts congruent to
$0, 1,$ or $2$ (mod $3$).

Exactly the same computational techniques we have described 
here then easily prove the following new result.  
Let $r_{m,i} (n)$ denote the number of Fibonacci partitions of
$n$ having the number of parts 
congruent to $i$ (mod $m$).
\begin{theorem}
We have $r_{3,i}(n) - r_{3,i+1} (n) \in \{ -1, 0, 1\}$ for $i \in \{0,1,2 \}$.
\end{theorem}

\begin{proof}
We use the following {\tt Walnut} commands to construct the automata
and matrices:
\begin{verbatim}
reg three1 {0,1} "0*(10*10*10*)*":
def fib3 "?msd_fib $three1(x) & $berst(x,y)":
eval fib3m y "?msd_fib $fib3(x,y)":
\end{verbatim}
Once we have a linear representation $(v, \mu, w)$
for $r_{3,0} (n)$, we can
easily construct ones for $r_{3,1} (n)$ and $r_{3,2} (n)$, by
modifying the final states specified in the vector $w$, and
then one can easily construct linear representations for
the difference $r_{3,i} (n) - r_{3,i+1} (n)$ for $i = 0, 1, 2$.
Then we proceed as before, minimizing the linear representations,
and using the ``semigroup trick'' to prove finiteness.  In this
case the size of the semigroup is $61$.
\end{proof}

However, for number of parts modulo $4$, the boundedness
property no longer holds.
The following result is also easily proved by our method:
\begin{theorem}
Let $d(n) = r_{4,0}(n) - r_{4,2} (n)$.  Then $d(n) = -16^k$
for $n = [(100)^{8k+1}1]_F$ and $d(n) = 4\cdot 16^k$
for $n = [(100)^{8k+5}1]_F$.
\end{theorem}

\begin{proof}
We use the same ideas.  This time the minimized linear
representation for $d(n)$ has rank $8$:
\begin{align*}
v &= \left[ \begin{smallarray}{cccccccc} 
1&0&0&0&0&0&0&0
\end{smallarray} \right]; \quad &
\mu(0) &= \left[ \begin{smallarray}{rrrrrrrr}
1& 0& 0& 0& 0& 0& 0& 0 \\
0& 0& 1& 0& 0& 0& 0& 0 \\
0& 0& 0& 1& 0& 0& 0& 0 \\
0& 0& 0& 0& 0& 1& 0& 0 \\
 0& 0&-1& 0& 0& 1& 0& 0\\
-1&-1& 0& 0& 1& 0&-1& 0\\
-1& 0&-1&-1& 0& 1& 1&-1\\
-1& 0&-1& 0& 0& 0& 0& 0
\end{smallarray}\right]; \quad &
\mu(1) &= \left[ \begin{smallarray}{cccccccc}
0&1&0&0&0&0&0&0\\
0&0&0&0&0&0&0&0\\
0&0&0&0&1&0&0&0\\
0&0&0&0&0&0&1&0\\
0&0&0&0&0&0&0&0\\
0&0&0&0&0&0&0&1\\
0&0&0&0&0&0&0&0\\
0&0&0&0&0&0&0&0
\end{smallarray}\right]; \quad &
w &= \left[ \begin{smallarray}{c}
1 \\
0 \\
0 \\
-1\\
-1\\
-1\\
-1\\
-1
\end{smallarray}\right].
\end{align*}
It is now easy to prove by induction that
$$\mu( (100)^{8n+5} 1) =
\left[ \begin{smallarray}{cccccccc}
0& 0& 0& 0& 0& 0& -4\cdot 16^k& 0 \\
0& 0& 0& 0& 0& 0& 0& 0 \\
0& 4 \cdot 16^k& 0& 0& 0& 0& 4 \cdot 16^k& 0 \\
0& 8 \cdot 16^k& 0& 0& 0& & 8 \cdot 16^k& 0 \\
 0& 0& 0& 0& 0& 0& 0& 0\\
 0& 4 \cdot 16^k& 0& 0& 0& 0& 4 \cdot 16^k& 0\\
 0& 0& 0& 0& 0& 0& 0& 0\\
 0& 0& 0& 0& 0& 0& 0& 0
\end{smallarray}\right],
$$ 
from which the claim
$d(n) = 4\cdot 16^k$
for $n = [(100)^{8k+5}1]_F$ follows immediately.
The other is handled similarly.
\end{proof}

\end{document}